\documentclass[11pt]{amsart}

\usepackage{amssymb,enumitem,mathtools}
\usepackage{hyperref}
\usepackage[margin=1.5in]{geometry}

\makeatletter
\newcommand{\vast}{\bBigg@{3}}
\makeatother

\newcommand{\vastl}{\mathopen\vast}
\newcommand{\vastr}{\mathclose\vast}

\DeclarePairedDelimiter{\abs}{\lvert}{\rvert}
\DeclarePairedDelimiter{\ceil}{\lceil}{\rceil}

\newtheorem{theorem}{Theorem}
\newtheorem{corollary}[theorem]{Corollary}
\newtheorem{lemma}[theorem]{Lemma}
\newtheorem{proposition}[theorem]{Proposition}

\theoremstyle{remark}
\newtheorem{example}[theorem]{Example}

\numberwithin{theorem}{section}

\title[Rational numbers with odd greedy expansion of fixed length]{Rational numbers with odd greedy expansion\\ of fixed length}

\author{Joel Louwsma}
\address[J.~Louwsma]{Department of Mathematics, Niagara University, Niagara University, NY 14109, USA}
\email{jlouwsma@niagara.edu}

\author{Joseph Martino}
\address[J.~Martino]{Department of Mathematics, Niagara University, Niagara University, NY 14109, USA}
\email{jmartino2345@gmail.com}

\begin{document}

\begin{abstract}
Given a positive rational number $n/d$ with $d$ odd, its odd greedy expansion starts with the largest odd denominator unit fraction at most $n/d$, adds the largest odd denominator unit fraction so the sum is at most $n/d$, and continues as long as the sum is less than $n/d$. It is an open question whether this expansion always has finitely many terms. Given a fixed positive integer~$n$, we find all reduced fractions with numerator~$n$ whose odd greedy expansion has length~$2$. Given $m-1$ odd positive integers, we find all rational numbers whose odd greedy expansion has length~$m$ and begins with these numbers as denominators. Given $m-2$ compatible odd positive integers, we find an infinite family of rational numbers whose odd greedy expansion has length~$m$ and begins with these numbers as denominators. 
\end{abstract}

\maketitle

\section{Introduction}

This paper studies expansions of positive rational numbers as sums of unit fractions with odd denominators. A \emph{unit fraction} is a fraction of the form $1/x$, where $x$ is a positive integer. The ancient Egyptians wrote rational numbers as sums of distinct unit fractions; for more about this history, see~\cite{RS}. Every positive rational number can be written as such a sum. One way to see this is by using the \emph{greedy algorithm} described by Fibonacci in his manuscript \emph{Liber Abaci} (see \cite[Chapter~7]{Fibonacci}). This algorithm begins with the largest unit fraction less than or equal to the initial rational number, adds the largest unit fraction so that the sum is less than or equal to the initial rational number, and continues until the sum equals the rational number. It always terminates and produces the \emph{greedy Egyptian expansion} of the rational number. For example, $7/15=1/3+1/8+1/120$. This algorithm was rediscovered by Sylvester~\cite{Sylvester} and others.

We study the related \emph{odd greedy algorithm}, which begins with a positive fraction with odd denominator and uses only unit fractions with odd denominators. The expansion it gives is called the \emph{odd greedy expansion} of the initial rational number. For example, $7/15=1/3+1/9+1/45$. Starke~\cite{Starke} proposed showing that every rational number with odd denominator is the sum of finitely many distinct unit fractions with odd denominators at least~$3$, and Stewart~\cite{Stewart} and independently Breusch~\cite{Breusch} did this, but neither of their proofs use odd greedy expansions. Indeed, whether the odd greedy expansion always has finitely many terms is a well-known open problem recorded by Guy \cite[Section~D11]{Guy2004} and Klee--Wagon \cite[Problem~15]{KW}. Eppstein~\cite{Eppstein} gives a heuristic argument for why the answer is likely to be positive. Wagon \cite[Section~15.2]{Wagon} gives several examples, due both to himself and others, of fractions whose odd greedy expansion has many terms and very large denominators. Brown~\cite{Brown} shows how to construct fractions whose odd greedy expansions have arbitrarily many terms. Pihko~\cite{Pihko2010} shows how to construct fractions for which the sequence of numerators of remainders under the odd greedy algorithm grows by~$1$ for arbitrarily many steps. 

One can define an odd greedy expansion by either allowing or prohibiting the term $1/1$ and either allowing or prohibiting repetition of terms; there is some variation in the literature in these regards. We choose to permit the term $1/1$ and permit repetition of terms. Precisely, given a positive rational number $n/d$, we construct $x_i$ recursively by letting $x_i=1$ when 
\[
\frac{n}{d}-\sum_{j=1}^{i-1}x_j\geq1
\]
and otherwise letting $x_i$ be the unique odd positive integer for which 
\[
\frac{1}{x_i}\leq\frac{n}{d}-\sum_{j=1}^{i-1}x_j<\frac{1}{x_i-2}.
\]
Since the term $1/1$ can only occur when $n/d\geq1$ and repeated terms can only occur when $n/d\geq2/3$, there is no difference between the various notions of odd greedy expansion for rational numbers less than $2/3$, which is the case of primary interest to us. 

Our focus is on the length of odd greedy expansions, i.e.\ the number of steps of the odd greedy algorithm. We take two perspectives. In the first of these, considered in Section~\ref{sec:fixednumerator}, we fix the numerator of a fraction and ask which denominators give rise to odd greedy expansions of length~$2$. For a fixed numerator, the main result of this section (Theorem~\ref{thm:length2}) finds all reduced fractions with that numerator whose odd greedy expansion has length~$2$. 

In the second perspective, considered in Section~\ref{sec:fixedm-1}, we fix odd positive integers $x_1,\dotsc,x_{m-1}$ and ask for which odd positive integers~$x_m$ we have that $x_1,\dotsc,x_{m}$ are the denominators of an odd greedy expansion. The main result of this section (Theorem~\ref{thm:fixedm-1}) finds all rational numbers whose odd greedy expansion has length~$m$ and begins with denominators $x_1,\dotsc,x_{m-1}$. In the special case of $m=2$, Corollary~\ref{cor:fixed1length2} finds all rational numbers whose odd greedy expansion has length~$2$ and begins with fixed denominator~$x_1$.

The fractions produced in Section~\ref{sec:fixedm-1} may or may not be in reduced form, and Section~\ref{sec:fixedm-1reduced} considers how they can reduce. Theorem~\ref{thm:commonfactorbounds} constrains the possible greatest common divisors of their numerators and denominators, and Theorem~\ref{thm:extremalvaluationglobal} determines when the bounds of Theorem~\ref{thm:commonfactorbounds} are attained. Along the way, Proposition~\ref{prop:fixed1length2reduced} finds all reduced forms of rational numbers whose odd greedy expansion has length~$2$ and begins with fixed denominator~$x_1$.

Finally, in Section~\ref{sec:fixedm-2}, we fix odd positive integers $x_1,\dotsc,x_{m-2}$ and ask for which odd integers $x_{m-1}$ and~$x_m$ we have that $x_1,\dotsc,x_{m}$ are the denominators of an odd greedy expansion. Given compatible $x_1,\dotsc,x_{m-2}$, the main result of this section (Theorem~\ref{thm:fixedm-2}) produces an infinite family of rational numbers whose odd greedy expansion has length~$m$ and begins with denominators $x_1,\dotsc,x_{m-2}$. 

Before proceeding, we introduce some notation that will be used throughout the paper. Given a positive integer~$m$, let $[m]$ denote the set $\{1,\dotsc,m\}$. 

Given a positive integer~$m$, variables $x_1,\dotsc,x_m$, and an integer~$k$ with $k\leq m$, let $\sigma_k(x_1,\dotsc,x_m)$ be the \emph{elementary symmetric polynomial} of degree~$k$ in these variables, i.e.\ the sum of all products of~$k$ of the $m$~variables. More precisely, 
\[
\sigma_k(x_1,\dotsc,x_m)\coloneqq\sum_{\substack{I\subseteq[m]\\\abs{I}=k}}\prod_{i\in I}x_i.
\]
By definition, $\sigma_0(x_1,\dotsc,x_m)=1$ and $\sigma_k(x_1,\dotsc,x_m)=0$ whenever $k<0$.

Given an integer~$n$ and a prime number~$p$, let $v_p(n)$ denote the \emph{$p$-adic valuation} of~$n$, i.e.\ the exponent of the largest power of~$p$ that divides~$n$. 

\section{Length \texorpdfstring{$2$}{2} with fixed numerator}\label{sec:fixednumerator}

In this section, we characterize the rational numbers whose odd greedy expansion has length~$2$ by fixing a numerator and determining which reduced fractions with that numerator have odd greedy expansion of length~$2$. We begin by observing that a rational number that is the sum of an even number of unit fractions with odd denominators can only be represented by a fraction with even numerator.

\begin{proposition}\label{prop:lengthparity}
Let $m$ be an even nonnegative integer. If a rational number is the sum of $m$ unit fractions with odd denominators, then every fraction representing this rational number has even numerator.
\end{proposition}

\begin{proof}
If a rational number is the sum of unit fractions with odd denominators $x_1,\dotsc,x_m$, then it is
\[
\sum_{i=1}^m\frac{1}{x_i}=\frac{\sigma_{m-1}(x_1,\dotsc, x_m)}{x_1\dotsm x_m}.
\]
The numerator $\sigma_{m-1}(x_1,\dotsc, x_m)$ is a sum of $m$~terms, each of which is odd, so hence it is even. The denominator $x_1\dotsm x_m$ is a product of odd integers, so it is odd. Since the denominator is odd, factors of~$2$ can never be canceled from the numerator, so every fraction representing this rational number has even numerator.
\end{proof}

We now find all rational numbers whose odd greedy expansion has length~$2$. Proposition~\ref{prop:lengthparity} allows us to only consider fractions with even numerators. We also restrict attention to fractions with odd denominators as fractions with even denominators are never reduced. 

\begin{proposition}\label{prop:length2}
Let $n$ be an even positive integer. The fractions with numerator~$n$ and odd positive denominators that have odd greedy expansion of length~$2$ are exactly those of the form 
\[
\frac{n}{n\bigl(\prod_{\substack{i=1}}^s p_i^{a_i}\bigr)(1+2t)-r},
\]
where $r$ is any odd positive integer less than $2n$, where $p_1,\dotsc,p_s$ are the prime divisors of~$r$, where $t$ is any nonnegative integer, and where
\[
a_i=\max\biggl\{\ceil[\bigg]{\frac{v_{p_i}(r)-v_{p_i}(n)}{2}},0\biggr\}.
\]
\end{proposition}

\begin{proof}
An expansion
\[
\frac{n}{d} = \frac{1}{x_1} + \frac{1}{x_2}
\]
is equivalent to
\[
\frac{1}{x_2} = \frac{nx_1-d}{dx_1}.
\]
Let $r=nx_1-d$, which must be positive as $d$, $x_1$, and~$x_2$ are all positive. This also means $d=nx_1-r$. Since $n$ is even, we have that $d$ and~$r$ have the same parity, so we may assume both are odd. We know $x_1$ will be the first denominator in the odd greedy expansion of $n/d$ if and only if
\[
\frac{1}{x_1-2}>\frac{n}{d}.
\]
This is equivalent to
\[
x_i-2<\frac{d}{n}=\frac{nx_1-r}{n}=x_1-\frac{r}{n},
\]
which is equivalent to $r<2n$.

We also have
\[
x_2 = \frac{dx_1}{nx_1-d} = \frac{(nx_1-r)x_1}{r} = \frac{nx_1^2}{r} - x_1.
\]
Since $n$ is even and $r$ and~$x_1$ are odd, $x_2$ will be an odd integer if and only if $nx_1^2$ is divisible by~$r$. This will happen if and only if $x_1$ is an odd integer such that, for all $i\in[s]$, we have $v_{p_i}(r)\leq v_{p_i}(nx_1^2)=v_{p_i}(n)+2v_{p_i}(x_1)$, which is equivalent to
\[
v_{p_i}(x_1)\geq\frac{v_{p_i}(r)-v_{p_i}(n)}{2}.
\]
Since $v_{p_i}(x_1)$ must be a nonnegative integer, this is equivalent to requiring that
\[
v_{p_i}(x_1)\geq\max\biggl\{\ceil[\bigg]{\frac{v_{p_i}(r)-v_{p_i}(n)}{2}},0\biggr\}=a_i
\]
for all $i\in[s]$. This is true if and only if $x_1$ is divisible by $\prod_{\substack{i=1}}^s p_i^{a_i}$. Since $x_1$ must be an odd positive integer, this is equivalent to having
\[
x_1 = \Biggl(\prod_{\substack{i=1}}^s p_i^{a_i}\Biggr)(1+2t)
\]
for some nonnegative integer~$t$. Substituting into $d=nx_1-r$, this means $n/d$ has odd greedy expansion of length~$2$ if and only if it is of the form 
\[
\frac{n}{n\bigl(\prod_{\substack{i=1}}^s p_i^{a_i}\bigr)(2t+1)-r}.\qedhere
\]
\end{proof}

We next determine when the fractions produced in Proposition~\ref{prop:length2} are in reduced form.

\begin{theorem}\label{thm:length2}
Let $n$ be an even positive integer. The fractions in reduced form with numerator~$n$ that have odd greedy expansion of length~$2$ are exactly those of the form 
\[
\frac{n}{n\Bigl(\prod_{\substack{i=1}}^s p_i^{\ceil{v_{p_i}(r)/2}}\Bigr)(1+2t)-r},
\]
where $r$ is any odd positive integer that is coprime to~$n$ and less than $2n$, where $p_1,\dotsc,p_s$ are the prime divisors of~$r$, and where $t$ is any nonnegative integer.
\end{theorem}

\begin{proof}
A rational number in reduced form that has odd greedy expansion of length~$2$ necessarily has even positive numerator by Proposition~\ref{prop:lengthparity} and necessarily has odd positive denominator because it is reduced, so therefore it must be of the form in Proposition~\ref{prop:length2}. Thus it remains to determine when fractions of this form are reduced. By the Euclidean algorithm, we have 
\[
\gcd\Biggl(n,n\Biggl(\prod_{\substack{i=1}}^s p_i^{a_i}\Biggr)(2t+1)-r\Biggr)=\gcd(n,r),
\]
so therefore such fractions are in reduced form exactly when $r$ is coprime to~$n$. In this case, we also have that $v_{p_i}(n)=0$ for all $i\in[s]$, so the expression for~$a_i$ in Proposition~\ref{prop:length2} reduces to 
\[
a_i=\ceil[\bigg]{\frac{v_{p_i}(r)}{2}}.\qedhere
\]
\end{proof}

We give examples of the families of rational numbers produced by Theorem~\ref{thm:length2}.

\begin{example}\label{ex:length2fixedn1}
Suppose $n=2$. The possible values of~$r$ are $1$ and~$3$, and each of these gives a family of reduced fractions with numerator~$2$ and odd greedy expansion of length~$2$. In the case $r=1$, this produces the following family of rational numbers with odd greedy expansion of length~$2$:
\[
\frac{2}{2(1)(1+2t)-1} = \frac{2}{1 + 4t}.
\]
\end{example}

\begin{example} \label{ex:length2fixedn2}
Suppose $n=6$. The possible values of~$r$ are $1$, $5$, $7$, and~$11$, and each of these gives a family of reduced fractions with numerator~$6$ and odd greedy expansion of length~$2$. In the case $r=5$, we have
\[
\ceil[\bigg]{\frac{v_5(r)}{2}} = \ceil[\bigg]{\frac{v_5(5)}{2}} = 
\ceil[\bigg]{\frac{1}{2}} = 1.
\]
This produces the following family of rational numbers with odd greedy expansion of length~$2$:
\[
\frac{6}{6(5^1)(1+2t)-5} = \frac{6}{25 + 60t}.
\]
\end{example}

\section{Length \texorpdfstring{$m$}{m} with fixed \texorpdfstring{$x_1,\dotsc,x_{m-1}$}{x\_1,...,x\_{m-1}}}\label{sec:fixedm-1}

In this section, we find all rational numbers whose odd greedy expansion has length~$m$ and begins with $m-1$ fixed positive odd denominators. We begin with a proposition that characterizes when a finite list of odd positive integers arises as the denominators of an odd greedy expansion. 

\begin{proposition}\label{prop:expansionequivalence}
Let $m$ be a positive integer, and let $x_1,\dotsc,x_m$ be odd positive integers. The following statements are equivalent:
\begin{enumerate}[label=\textup{(\alph*)},ref=\textup{\alph*}]
\item The integers $x_1,\dotsc,x_m$ are the denominators of the odd greedy expansion of the sum of their reciprocals. \label{item:expansion}
\item For all positive integers $i$ and~$k$ with $i\leq k\leq m$, we have \label{item:inequality}
\[
2\sigma_{k-i}(x_i,\dotsc, x_{k})>x_i^2\sigma_{k-i-1}(x_{i+1},\dotsc, x_{k}).
\]
\item For all positive integers $i$ and~$k$ with $i<k\leq m$, we have \label{item:x_kbound}
\begin{equation}\label{eq:x_kbound}
x_k>\frac{(x_i-2)x_{i}\dotsm x_{k-1}}{2\sigma_{k-i-1}(x_i,\dotsc, x_{k-1})-x_i^2\sigma_{k-i-2}(x_{i+1},\dotsc, x_{k-1})}.
\end{equation}
\end{enumerate}
\end{proposition}

\begin{proof}
We first show~(\ref{item:expansion}) implies~(\ref{item:inequality}). Suppose $x_1,\dotsc,x_m$ are the denominators of the odd greedy expansion of the sum of their reciprocals. Fix positive integers $i$ and~$k$ with $i\leq k\leq m$. In the case when $x_i\geq3$, since $x_i$ occurs as the $i$th denominator of the odd greedy expansion of $\sum_{j=1}^m1/x_j$, we have
\[
\frac{1}{x_i-2}>\sum_{j=i}^m\frac{1}{x_j}\geq\sum_{j=i}^k\frac{1}{x_j}=\frac{\sigma_{k-i}(x_i,\dotsc, x_{k})}{x_i\dotsm x_{k}}.
\]
This implies that
\[
\begin{split}
x_i\dotsm x_{k}&>x_i\sigma_{k-i}(x_i,\dotsc, x_{k})-2\sigma_{k-i}(x_i,\dotsc, x_{k})\\
&=x_i^2\sigma_{k-i-1}(x_{i+1},\dotsc, x_{k})+x_i\dotsm x_{k}-2\sigma_{k-i}(x_i,\dotsc, x_{k}),
\end{split}
\]
and hence that
\[
2\sigma_{k-i}(x_i,\dotsc, x_{k})>x_i^2\sigma_{k-i-1}(x_{i+1},\dotsc, x_{k}).
\]
When $x_i=1$, we have
\[
2\sigma_{k-i}(x_i,\dotsc, x_{k})=2x_i\sigma_{k-i-1}(x_{i+1},\dotsc, x_{k})+2x_{i+1}\dotsm x_k>x_i^2\sigma_{k-i-1}(x_{i+1},\dotsc, x_{k}).
\]

We next show~(\ref{item:inequality}) implies~(\ref{item:expansion}). Suppose
\[
2\sigma_{k-i}(x_i,\dotsc, x_{k})>x_i^2\sigma_{k-i-1}(x_{i+1},\dotsc, x_{k})
\]
for all positive integers $i$ and~$k$ with $i\leq k\leq m$. In the special case when $k=m$, this becomes $2\sigma_{m-i}(x_i,\dotsc, x_m)>x_i^2\sigma_{m-i-1}(x_{i+1},\dotsc, x_m)$, which gives 
\[
\begin{split}
x_i\dotsm x_m&>x_i^2\sigma_{m-i-1}(x_{i+1},\dotsc, x_m)+x_i\dotsm x_m-2\sigma_{m-i}(x_i,\dotsc, x_m)\\
&=x_i\sigma_{m-i}(x_i,\dotsc, x_m)-2\sigma_{m-i}(x_i,\dotsc, x_m).
\end{split}
\]
In the case when $x_i\geq3$, this becomes
\[
\frac{1}{x_i-2}>\sum_{j=i}^m\frac{1}{x_j}=\frac{\sigma_{m-i}(x_i,\dotsc, x_m)}{x_i\dotsm x_m},
\]
meaning that $x_i$ is the $i$th denominator of the odd greedy expansion of $\sum_{j=1}^m1/x_j$. When $x_i=1$, we have
\[
\sum_{j=i}^m\frac{1}{x_j}\geq1,
\]
so $x_i$ is also the $i$th denominator of the odd greedy expansion of $\sum_{j=1}^m1/x_j$. Thus $x_1,\dotsc,x_m$ are the denominators of the odd greedy expansion of $\sum_{j=1}^m1/x_j$.

We next show~(\ref{item:inequality}) implies~(\ref{item:x_kbound}). Suppose 
\[
2\sigma_{k-i}(x_i,\dotsc, x_{k})>x_i^2\sigma_{k-i-1}(x_{i+1},\dotsc, x_{k})
\]
for all positive integers $i$ and~$k$ with $i\leq k\leq m$. Under the assumption $i<k$, this implies
\[
2x_i\dotsm x_{k-1}+2\sigma_{k-i-1}(x_i,\dotsc, x_{k-1})x_k>x_i(x_i\dotsm x_{k-1})+x_i^2\sigma_{k-i-2}(x_{i+1},\dotsc, x_{k-1})x_k,
\]
which implies
\[
\bigl(2\sigma_{k-i-1}(x_i,\dotsc, x_{k-1})-x_i^2\sigma_{k-i-2}(x_{i+1},\dotsc, x_{k-1})\bigr)x_k>(x_i-2)x_{i}\dotsm x_{k-1}.
\]
Since~(\ref{item:inequality}) with $k$ replaced by $k-1$ gives that the coefficient of~$x_k$ is positive, it follows that~\eqref{eq:x_kbound} holds.

Finally, we show~(\ref{item:x_kbound}) implies~(\ref{item:inequality}). We do this by showing by induction on~$k$ that $2\sigma_{k-i}(x_i,\dotsc, x_k)>x_i^2\sigma_{k-i-1}(x_{i+1},\dotsc, x_k)$ for all positive integers $i$ and~$k$ with $i\leq k\leq m$. In the base case $k=1$, we must have $i=1$, in which case the statement says $2>0$. In the inductive step, if $i=k$, then the statement says $2>0$. If $i<k$, then the inductive hypothesis gives that the denominator of the fraction on the right of~\eqref{eq:x_kbound} is positive, so we can multiply through by it to get 
\[
\bigl(2\sigma_{k-i-1}(x_i,\dotsc, x_{k-1})-x_i^2\sigma_{k-i-2}(x_{i+1},\dotsc, x_{k-1})\bigr)x_k>(x_i-2)x_{i}\dotsm x_{k-1}.
\]
This implies
\[
2x_i\dotsm x_{k-1}+2\sigma_{k-i-1}(x_i,\dotsc, x_{k-1})x_k>x_i(x_i\dotsm x_{k-1})+x_i^2\sigma_{k-i-2}(x_{i+1},\dotsc, x_{k-1})x_k,
\]
which implies
\[
2\sigma_{k-i}(x_i,\dotsc, x_k)>x_i^2\sigma_{k-i-1}(x_{i+1},\dotsc, x_k).\qedhere
\]
\end{proof}

We now use Proposition~\ref{prop:expansionequivalence} to find all rational numbers whose odd greedy expansion has length~$m$ and begins with $m-1$ fixed positive odd denominators.

\begin{theorem}\label{thm:fixedm-1}
Let $m$ be a positive integer, and let $x_1,\dotsc,x_{m-1}$ be odd positive integers that satisfy inequality~\eqref{eq:x_kbound} for all positive integers $i$ and~$k$ with $i<k\leq m-1$. If $b$ is the smallest odd positive integer greater than
\[
\max_{i\in[m-1]}\biggl\{\frac{(x_i-2)x_{i}\dotsm x_{m-1}}{2\sigma_{m-i-1}(x_i,\dotsc, x_{m-1})-x_i^2\sigma_{m-i-2}(x_{i+1},\dotsc, x_{m-1})}\biggr\},
\]
then the rational numbers whose odd greedy expansion has length~$m$ and begins with denominators $x_1,\dotsc,x_{m-1}$ are exactly those of the form
\[
\frac{\sigma_{m-1}(x_1,\dotsc ,x_{m-1},b)+2\sigma_{m-2}(x_1,\dotsc ,x_{m-1})t}{x_1\dotsm x_{m-1}b+2x_1\dotsm x_{m-1}t},
\]
where $t$ is any nonnegative integer.
\end{theorem}

\begin{proof}
Since we have assumed inequality~\eqref{eq:x_kbound} holds whenever $i<k\leq m-1$, according to Proposition~\ref{prop:expansionequivalence} a given~$x_m$ will lead to an odd greedy expansion of length~$m$ if and only if~\eqref{eq:x_kbound} also holds for all $i\in[m-1]$ when $k=m$. Since $b$ is the smallest odd integer greater than all the lower bounds on~$x_m$ given by~\eqref{eq:x_kbound}, all of these inequalities will be satisfied exactly when $x_m=b+2t$ for some nonnegative integer~$t$. The rational number is then
\[
\begin{split}
\sum_{j=1}^{m}\frac{1}{x_j}&=\frac{\sigma_{m-1}(x_1,\dotsc ,x_{m-1},b+2t)}{x_1\dotsm x_{m-1}(b+2t)}\\
&=\frac{x_1\dotsm x_{m-1}+\sigma_{m-2}(x_1,\dotsc ,x_{m-1})b+2\sigma_{m-2}(x_1,\dotsc ,x_{m-1})t}{x_1\dotsm x_{m-1}b+2x_1\dotsm x_{m-1}t}\\
&=\frac{\sigma_{m-1}(x_1,\dotsc ,x_{m-1},b)+2\sigma_{m-2}(x_1,\dotsc ,x_{m-1})t}{x_1\dotsm x_{m-1}b+2x_1\dotsm x_{m-1}t}.\qedhere
\end{split}
\]
\end{proof}

In the special case of $m=2$, we can find $b$ more explicitly and thus get a more explicit description of the rational numbers whose odd greedy expansion has length~$2$ and begins with a given positive odd denominator.

\begin{corollary}\label{cor:fixed1length2}
Let $x_1$ be an odd positive integer. The rational numbers whose odd greedy expansion has length~$2$ and begins with denominator~$x_1$ are exactly those of the form
\[
\frac{(x_1^2+3)/2+2t}{(x_1^3+3x_1)/2-x_1^2+2x_1t},
\]
where $t$ is any nonnegative integer.
\end{corollary}

\begin{proof}
In Theorem~\ref{thm:fixedm-1}, we have that $b$ is the smallest odd positive integer greater than
\[
\frac{x_1^2-2x_1}{2}=\frac{x_1^2-2x_1+1-1}{2}=\frac{(x_1-1)^2}{2}-\frac{1}{2}.
\]
Since $(x_1-1)^2/2$ is always an even number, this means 
\[
b=\frac{(x_1-1)^2}{2}+1=\frac{x_1^2-2x_1+1+2}{2}=\frac{x_1^2+3}{2}-x_1.
\]
Therefore Theorem~\ref{thm:fixedm-1} gives that the rational numbers whose odd greedy expansion has length~$2$ and begins with denominator~$x_1$ are exactly those of the form
\[
\begin{split}
\frac{\sigma_1(x_1,b)+2\sigma_0(x_1)t}{x_1b+2x_1t}&=\frac{x_1+b+2t}{x_1b+2tx_1}\\
&=\frac{x_1+(x_1^2+3)/2-x_1+2t}{x_1\bigl((x_1^2+3)/2-x_1\bigr)+2x_1t}\\
&=\frac{(x_1^2+3)/2+2t}{(x_1^3+3x_1)/2-x_1^2+2x_1t}.\qedhere
\end{split}
\]
\end{proof}

We give several examples, first in the case of length~$2$.

\begin{example}\label{ex:fixed1length2a}
When $m=2$ and $x_1=3$, Corollary~\ref{cor:fixed1length2} says the rational numbers whose odd greedy expansion has length~$2$ with first denominator~$3$ are exactly those of the form
\[
\frac{(x_1^2+3)/2+2t}{(x_1^3+3x_1)/2-x_1^2+2x_1t}=\frac{(3^2+3)/2+2t}{(3^3+3\cdot3)/2-3^2+2\cdot 3t}=\frac{6+2t}{9+6t}.
\]
\end{example}

\begin{example}\label{ex:fixed1length2b}
When $m=2$ and $x_1=5$, Corollary~\ref{cor:fixed1length2} says the rational numbers whose odd greedy expansion has length~$2$ with first denominator~$5$ are exactly those of the form
\[
\frac{(x_1^2+3)/2+2t}{(x_1^3+3x_1)/2-x_1^2+2x_1t}=\frac{(5^2+3)/2+2t}{(5^3+3\cdot5)/2-5^2+2\cdot 5t}=\frac{14+2t}{45+10t}.
\]
\end{example}

We also give some examples in the case of length~$3$.

\begin{example}\label{ex:fixed2length3a}
Suppose $m=3$, $x_1=3$ and $x_2=5$. Since~\eqref{eq:x_kbound} is satisfied in the case $i=1$ and $k=2$, we can apply Theorem~\ref{thm:fixedm-1}. When $k=3$, the bounds of~\eqref{eq:x_kbound} give that $x_3>15/7$ and $x_3>15/2$, so therefore $b=9$. Thus the rational numbers whose odd greedy expansion has length~$3$ and begins with denominators $3$ and~$5$ are exactly those of the form 
\[
\frac{\sigma_2(3,5,9)+2\sigma_1(3,5)t}{3\cdot5\cdot9+2\cdot3\cdot5\cdot t}=\frac{87+16t}{135+30t}.
\]
\end{example}

\begin{example}\label{ex:fixed2length3}
Suppose $m=3$, $x_1=5$ and $x_2=9$. Since~\eqref{eq:x_kbound} is satisfied in the case $i=1$ and $k=2$, we can apply Theorem~\ref{thm:fixedm-1}. When $k=3$, the bounds of~\eqref{eq:x_kbound} give that $x_3>45$ and $x_3>63/2$, so therefore $b=47$. Thus the rational numbers whose odd greedy expansion has length $3$ and begins with denominators $5$ and~$9$ are exactly those of the form 
\[
\frac{\sigma_2(5,9,47)+2\sigma_1(5,9)t}{5\cdot9\cdot47+2\cdot5\cdot9\cdot t}=\frac{703+28t}{2115+90t}.
\]
\end{example}

\section{Reductions with length \texorpdfstring{$m$}{m} and fixed \texorpdfstring{$x_1,\dotsc,x_{m-1}$}{x\_1,...,x\_{m-1}}}\label{sec:fixedm-1reduced}

Unlike in Section~\ref{sec:fixednumerator}, where the families of fractions produced are either always in reduced form or never in reduced form, the families of fractions produced in Section~\ref{sec:fixedm-1} are commonly in reduced form for some values of~$t$ but not for other values of~$t$. In this section, we study the possible greatest common divisors of the numerator and denominator of 
\begin{equation}\label{eq:generalfraction}
\frac{\sigma_{m-1}(x_1,\dotsc,x_m)}{x_1\dotsm x_m}
\end{equation}
in terms of $x_1,\dotsc,x_{m-1}$.

\subsection{Constraints on greatest common divisors}
This subsection finds constraints on greatest common divisors of the numerator and denominator of~\eqref{eq:generalfraction} in terms of $p$-adic valuations of $x_1,\dotsc,x_{m-1}$. We begin with a lemma that finds an expression for a $p$-adic valuation of the numerator $\sigma_{m-1}(x_1,\dotsc,x_m)$. 

\begin{lemma}\label{lem:symmetricpolynomialvaluation}
Let $x_1,\dotsc,x_{m}$ be positive integers, and let $p$ be a prime number. If $W_p=\max_{i\in[m]}\{v_p(x_i)\}$, then
\[
v_p(\sigma_{m-1}(x_1,\dotsc,x_m))=v_p(x_1\dotsm x_m)-W_p+v_p\vastl(\sum_{i=1}^mp^{W_p-v_p(x_i)}\prod_{\substack{j=1\\j\neq i}}^m\frac{x_j}{p^{v_p(x_j)}}\vastr).
\]
\end{lemma}

\begin{proof}
We compute that
\[
\begin{split}
v_p(\sigma_{m-1}(x_1,\dotsc,x_m))&=v_p\vastl(\sum_{i=1}^m\prod_{\substack{j=1\\j\neq i}}^{m}{x_j}\vastr)\\
&=v_p\vastl(\sum_{i=1}^{m}p^{v_p(x_1\dotsm x_{m}/x_i)}\prod_{\substack{j=1\\j\neq i}}^{m}\frac{x_j}{p^{v_p(x_j)}}\vastr)\\
&=v_p\vastl(p^{v_p(x_1\dotsm x_{m})-W_p}\sum_{i=1}^{m}p^{W_p-v_p(x_i)}\prod_{\substack{j=1\\j\neq i}}^{m}\frac{x_j}{p^{v_p(x_j)}}\vastr)\\
&=v_p(x_1\dotsm x_m)-W_p+v_p\vastl(\sum_{i=1}^mp^{W_p-v_p(x_i)}\prod_{\substack{j=1\\j\neq i}}^m\frac{x_j}{p^{v_p(x_j)}}\vastr).\qedhere
\end{split}
\]
\end{proof}

We simplify the expression in Lemma~\ref{lem:symmetricpolynomialvaluation} depending on the size of the $p$-adic valuation of~$x_m$ relative to the $p$-adic valuations of $x_1,\dotsc,x_{m-1}$.

\begin{lemma}\label{lem:symmetricpolynomialvaluationcases}
Let $x_1,\dotsc,x_{m}$ be positive integers, let $p$ be a prime number, and let $V_p=\max_{i\in[m-1]}\{v_p(x_i)\}$. 
\begin{enumerate}[label=\textup{(\alph*)},ref=\textup{\alph*}]
\item If $v_p(x_m)\leq V_p$, then \label{item:symmetricpolynomialvaluationcase1}
\[
v_p(\sigma_{m-1}(x_1,\dotsc,x_m))=v_p(x_1\dotsm x_{m})-V_p+v_p\vastl(\sum_{i=1}^{m}p^{V_p-v_p(x_i)}\prod_{\substack{j=1\\j\neq i}}^{m}\frac{x_j}{p^{v_p(x_j)}}\vastr).
\]
\item If $v_p(x_m)\geq V_p$, then \label{item:symmetricpolynomialvaluationcase2}
\[
v_p(\sigma_{m-1}(x_1,\dotsc,x_m))=v_p(x_1\dotsm x_{m-1})+v_p\vastl(\sum_{i=1}^{m}p^{v_p(x_m)-v_p(x_i)}\prod_{\substack{j=1\\j\neq i}}^{m}\frac{x_j}{p^{v_p(x_j)}}\vastr).
\]
\item If $v_p(x_m)> V_p$, then $v_p(\sigma_{m-1}(x_1,\dotsc,x_m))=v_p(x_1\dotsm x_{m-1})$. \label{item:symmetricpolynomialvaluationcase3}
\end{enumerate}
\end{lemma}

\begin{proof}
As in Lemma~\ref{lem:symmetricpolynomialvaluation}, let $W_p=\max_{i\in[m]}\{v_p(x_i)\}$. When $v_p(x_m)\leq V_p$, we have $W_p=V_p$. Applying Lemma~\ref{lem:symmetricpolynomialvaluation} then yields~(\ref{item:symmetricpolynomialvaluationcase1}). When $v_p(x_m)\geq V_p$, we have $W_p=v_p(x_m)$. Applying Lemma~\ref{lem:symmetricpolynomialvaluation} then yields~(\ref{item:symmetricpolynomialvaluationcase2}). In the special case of~(\ref{item:symmetricpolynomialvaluationcase2}) in which $v_p(x_m)> V_p$, we have that $v_p(x_i)<v_p(x_m)$ for all $i\in[m-1]$. Therefore
\[
p^{v_p(x_m)-v_p(x_i)}\prod_{\substack{j=1\\j\neq i}}^{m}\frac{x_j}{p^{v_p(x_j)}}
\]
is divisible by~$p$ for all $i\neq m$ but not divisible by~$p$ when $i=m$. Hence 
\[
v_p\vastl(\sum_{i=1}^{m}p^{v_p(x_m)-v_p(x_i)}\prod_{\substack{j=1\\j\neq i}}^{m}\frac{x_j}{p^{v_p(x_j)}}\vastr)=0,
\]
showing~(\ref{item:symmetricpolynomialvaluationcase3}).
\end{proof}

We use Lemma~\ref{lem:symmetricpolynomialvaluationcases} to establish bounds on $p$-adic valuations of greatest common divisors of the numerator and denominator of~\eqref{eq:generalfraction} in terms of $x_1,\dotsc,x_{m-1}$.

\begin{theorem}\label{thm:commonfactorbounds}
Let $x_1,\dotsc,x_{m}$ be integers, and let 
\[
y=\gcd(\sigma_{m-1}(x_1,\dotsc,x_m),x_1\dotsm x_m).
\]
If $p$ is a prime number and $V_p=\max_{i\in[m-1]}\{v_p(x_i)\}$, then
\begin{enumerate}[label=\textup{(\alph*)},ref=\textup{\alph*}]
\item $v_p(y)\geq v_p(x_1\dotsm x_{m-1})-V_p$ and \label{item:commonfactorlowerbound}
\item $v_p(y)\leq v_p(x_1\dotsm x_{m-1})+V_p$. \label{item:commonfactorupperbound}
\end{enumerate}
\end{theorem}

\begin{proof}
Fix $p$. We first show~(\ref{item:commonfactorlowerbound}). We have 
\[
v_p(x_1\dotsm x_{m})\geq v_p(x_1\dotsm x_{m-1})\geq v_p(x_1\dotsm x_{m-1})-V_p.
\]
For all $j\in[m-1]$, we also have 
\[
v_p(x_1\dotsm x_{j-1}x_{j+1}\dotsm x_{m})\geq v_p(x_1\dotsm x_{m})-V_p
\geq v_p(x_1\dotsm x_{m-1})-V_p.
\]
This implies that
\[
\begin{split}
v_p(\sigma_{m-1}(x_1,\dotsc,x_m)) &\geq \min_{j\in[m]}\{v_p(x_1\dotsm x_{j-1}x_{j+1}\dotsm x_{m})\}\\
&\geq v_p(x_1\dotsm x_{m-1})-V_p.
\end{split}
\]
It follows that
\[
v_p(y)=\min\{v_p(\sigma_{m-1}(x_1,\dotsc,x_m)),v_p(x_1\dotsm x_m)\}\geq v_p(x_1\dotsm x_{m-1})-V_p,
\]
completing the proof of~(\ref{item:commonfactorlowerbound}).

To prove~(\ref{item:commonfactorupperbound}), we consider two cases: when $v_p(x_m)\leq V_p$ and when $v_p(x_m)>V_p$. When $v_p(x_m)\leq V_p$, we have $v_p(y)\leq v_p(x_1\dotsm x_m)\leq v_p(x_1\dotsm x_{m-1})+V_p$. When $v_p(x_m)>V_p$, Lemma \ref{lem:symmetricpolynomialvaluationcases}(\ref{item:symmetricpolynomialvaluationcase3}) gives that $v_p(\sigma_{m-1}(x_1,\dotsc,x_m))=v_p(x_1\dotsm x_{m-1})$. Therefore $v_p(y)=v_p(x_1\dotsm x_{m-1})\leq v_p(x_1\dotsm x_{m-1})+V_p$.
\end{proof}

We record a corollary in the case $m=2$. 

\begin{corollary}\label{cor:commonfactorboundslength2}
If $x_1$ and~$x_2$ are integers and $y=\gcd(x_1+x_2,x_1x_2)$, then $y$ divides~$x_1^2$.
\end{corollary}

\begin{proof}
When $m=2$, we have $V_p=v_p(x_1)$ in Theorem~\ref{thm:commonfactorbounds}. Part~(\ref{item:commonfactorupperbound}) then gives that $v_p(y)\leq 2v_p(x_1)=v_p(x_1^2)$. Since this is true for every prime number~$p$, it follows that $y$ divides~$x_1^2$.
\end{proof}

Continuing with the special case $m=2$, we can use Corollaries \ref{cor:fixed1length2} and~\ref{cor:commonfactorboundslength2} to find the reduced form of all rational numbers whose odd greedy expansion has length~$2$ and begins with a fixed odd denominator.

\begin{proposition}\label{prop:fixed1length2reduced}
Let $x_1$ be an odd positive integer. Any rational number whose odd greedy expansion has length~$2$ and begins with denominator~$x_1$ has reduced form of the form
\[
\frac{2\ceil[\Big]{\frac{x_1^2+3}{4y}}+2u}{2x_1\ceil[\Big]{\frac{x_1^2+3}{4y}}-\frac{x_1^2}{y}+2x_1u},
\]
where $y$ is any divisor of~$x_1^2$ and $u$ is any nonnegative integer.
\end{proposition}

\begin{proof}
We know from Corollary~\ref{cor:fixed1length2} that every rational number whose odd greedy expansion has length~$2$ and begins with~$x_1$ is of the form
\[
\frac{(x_1^2+3)/2+2t}{(x_1^3+3x_1)/2-x_1^2+2x_1t},
\]
and we know from Corollary~\ref{cor:commonfactorboundslength2} that the possible greatest common divisors of the numerator and denominator of this fraction are divisors of $x_1^2$. Let $y$ be a divisor of~$x_1^2$. We will show that the fractions from Corollary~\ref{cor:fixed1length2} for which $y$ is the greatest common divisor of the numerator and denominator have reduced form of the form given in the statement.

Since the numerator $(x_1^2+3)/2+2t$ is even and $y$ is odd, the values of the numerator that are divisible by~$y$ are precisely the multiples of $2y$ that are at least $(x_1^2+3)/2$. These are exactly the integers of the form
\[
2y\ceil[\bigg]{\frac{(x_1^2+3)/2}{2y}}+2yu = 2y\ceil[\bigg]{\frac{x_1^2+3}{4y}}+2yu,
\]
where $u$ is any nonnegative integer. Setting the numerator equal to this, we have 
\[
\frac{x_1^2+3}{2}+2t = 2y\ceil[\bigg]{\frac{x_1^2+3}{4y}}+2yu,
\]
and hence
\[
t=y\ceil[\bigg]{\frac{x_1^2+3}{4y}}-\frac{x_1^2+3}{4}+yu.
\]
Substituting this into the fraction, we have that it must be of the form 
\[
\frac{2y\ceil[\Big]{\frac{x_1^2+3}{4y}}+2yu}{\frac{x_1^3+3x_1}{2}-x_1^2+2x_1\Bigl(y\ceil[\Big]{\frac{x_1^2+3}{4y}}-\frac{x_1^2+3}{4}+yu\Bigr)}=\frac{2y\ceil[\Big]{\frac{x_1^2+3}{4y}}+2yu}{2x_1y\ceil[\Big]{\frac{x_1^2+3}{4y}}-x_1^2+2x_1yu}.
\]
Dividing the numerator and denominator by~$y$, we get that all fractions from Corollary~\ref{cor:fixed1length2} for which the greatest common divisor of the numerator and denominator is~$y$ have reduced form of the form
\[
\frac{2\ceil[\Big]{\frac{x_1^2+3}{4y}}+2u}{2x_1\ceil[\Big]{\frac{x_1^2+3}{4y}}-\frac{x_1^2}{y}+2x_1u}.\qedhere
\]
\end{proof}

We consider an example that finds the reduced forms of the fractions from Example~\ref{ex:fixed1length2b}.

\begin{example}
When $x_1=5$, the possible values of~$y$ are $1$, $5$, and $25$. Every rational number with odd greedy expansion of length~$2$ and first denominator~$5$ has reduced form of one of the following forms.
\begin{align*}
y=1: &\qquad \frac{2\ceil[\Big]{\frac{5^2+3}{4(1)}}+2u}{2(5)\ceil[\Big]{\frac{5^2+3}{4(1)}}-\frac{5^2}{1}+2(5)u} = \frac{14+2u}{45+10u}\\
y=5: &\qquad \frac{2\ceil[\Big]{\frac{5^2+3}{4(5)}}+2u}{2(5)\ceil[\Big]{\frac{5^2+3}{4(5)}}-\frac{5^2}{5}+2(5)u} = \frac{4+2u}{15+10u}\\
y=25: &\qquad \frac{2\ceil[\Big]{\frac{5^2+3}{4(25)}}+2u}{2(5)\ceil[\Big]{\frac{5^2+3}{4(25)}}-\frac{5^2}{25}+2(5)u} = \frac{2+2u}{9+10u}\\
\end{align*}
\end{example}

Note that the fraction in Proposition~\ref{prop:fixed1length2reduced} need not be reduced for all $y$ and~$u$; we only claim that the reduced form of every rational number whose odd greedy expansion has length~$2$ and begins with denominator~$x_1$ is of this form for some $y$ and~$u$.

\subsection{When constraints on greatest common divisors are sharp}
The bounds of Theorem~\ref{thm:commonfactorbounds} can be attained and hence cannot be improved in general. However, given $x_1,\dotsc,x_{m-1}$ there may or may not exist an $x_m$ for which the bounds are attained. In this subsection, we characterize when there is such an $x_m$, first working with one prime number and then extending to all prime numbers. Throughout, we let $y$ denote $\gcd(\sigma_{m-1}(x_1,\dotsc,x_m),x_1\dotsm x_m)$. 

\begin{proposition}\label{prop:extremalvaluationlocal}
Let $x_1,\dotsc,x_{m-1}$ be odd positive integers, let $p$ be a prime number that divides $x_i$ for some $i\in[m-1]$, and let $V_p=\max_{i\in[m-1]}\{v_p(x_i)\}$. The following statements are equivalent: 
\begin{enumerate}[label=\textup{(\alph*)},ref=\textup{\alph*}]
\item There exists an odd positive integer~$x_m$ with $v_p(y)=v_p(x_1\dotsm x_{m-1})-V_p$. \label{item:lowervaluationlocal}
\item There exists an odd positive integer~$x_m$ with $v_p(y)=v_p(x_1\dotsm x_{m-1})+V_p$. \label{item:uppervaluationlocal}
\item $p$ does not divide \label{item:valuationcriterionlocal}
\begin{equation}
\label{eq:extremalvaluationcondition}
\sum_{i=1}^{m-1}p^{V_p-v_p(x_i)}\prod_{\substack{j=1\\j\neq i}}^{m-1}\frac{x_j}{p^{v_p(x_j)}}.
\end{equation}
\end{enumerate}
\end{proposition}

\begin{proof}
We first prove that~(\ref{item:lowervaluationlocal}) is equivalent to~(\ref{item:valuationcriterionlocal}). Since
\[
y=\gcd(\sigma_{m-1}(x_1,\dotsc,x_m),x_1\dotsm x_m),
\]
we have
\[
v_p(y)=\min\{v_p(\sigma_{m-1}(x_1,\dotsc,x_m)),v_p(x_1\dotsm x_{m})\}.
\]
Since $V_p\geq1$, it is always the case that $v_p(x_1\dotsm x_{m})>v_p(x_1\dotsm x_{m-1})-V_p$, i.e.\ $v_p(x_1\dotsm x_{m})$ is always larger than the bound of Theorem~\ref{thm:commonfactorbounds}. Therefore the condition $v_p(y)=v_p(x_1\dotsm x_{m-1})-V_p$ is equivalent to
\[
v_p(\sigma_{m-1}(x_1,\dotsc,x_m))=v_p(x_1\dotsm x_{m-1})-V_p.
\]
By Lemma \ref{lem:symmetricpolynomialvaluationcases}(\ref{item:symmetricpolynomialvaluationcase1}), this holds if and only if neither $x_m$ nor
\[
\sum_{i=1}^{m}p^{V_p-v_p(x_i)}\prod_{\substack{j=1\\j\neq i}}^{m}\frac{x_j}{p^{v_p(x_j)}}
\]
is divisible by~$p$. When $v_p(x_m)=0$, the assumption $V_p\geq1$ implies that the $i=m$ term of this sum is divisible by~$p$, so therefore the above is equivalent to requiring that 
\[
\vastl(\sum_{i=1}^{m-1}p^{V_p-v_p(x_i)}\prod_{\substack{j=1\\j\neq i}}^{m-1}\frac{x_j}{p^{v_p(x_j)}}\vastr)x_m
\]
is not divisible by~$p$. Hence if $v_p(y)=v_p(x_1\dotsm x_{m-1})-V_p$ it must be the case that~\eqref{eq:extremalvaluationcondition} is not divisible by~$p$. Conversely, if~\eqref{eq:extremalvaluationcondition} is not divisible by~$p$, if we let $x_m$ be any odd positive integer that is not divisible by~$p$, we have that
\[
\vastl(\sum_{i=1}^{m-1}p^{V_p-v_p(x_i)}\prod_{\substack{j=1\\j\neq i}}^{m-1}\frac{x_j}{p^{v_p(x_j)}}\vastr)x_m
\]
is not divisible by~$p$. 

We next prove that~(\ref{item:uppervaluationlocal}) is equivalent to~(\ref{item:valuationcriterionlocal}). If $v_p(x_m)<V_p$, then
\[
v_p(y)\leq v_p(x_1\dotsm x_{m})<v_p(x_1\dotsm x_{m-1})+V_p.
\]
If $v_p(x_m)>V_p$, then it follows from Lemma \ref{lem:symmetricpolynomialvaluationcases}(\ref{item:symmetricpolynomialvaluationcase3}) that 
\[
v_p(y)\leq v_p(\sigma_{m-1}(x_1,\dotsc,x_m))=v_p(x_1\dotsm x_{m-1})<v_p(x_1\dotsm x_{m-1})+V_p.
\]
Together these show~(\ref{item:uppervaluationlocal}) is only possible if $v_p(x_m)=V_p$, so it suffices to show that~(\ref{item:uppervaluationlocal}) is equivalent to~(\ref{item:valuationcriterionlocal}) under the assumption $v_p(x_m)=V_p$. In this case, we have
\[
v_p(x_1\dotsm x_{m})=v_p(x_1\dotsm x_{m-1})+V_p.
\]
Therefore $v_p(y)=v_p(x_1\dotsm x_{m-1})+V_p$ if and only if 
\[
v_p(\sigma_{m-1}(x_1,\dotsc,x_m))\geq v_p(x_1\dotsm x_{m-1})+V_p.
\]
By Lemma \ref{lem:symmetricpolynomialvaluationcases}(\ref{item:symmetricpolynomialvaluationcase2}), this happens if and only if 
\[
\sum_{i=1}^{m}p^{V_p-v_p(x_i)}\prod_{\substack{j=1\\j\neq i}}^{m}\frac{x_j}{p^{v_p(x_j)}}=\prod_{i=1}^{m-1}\frac{x_i}{p^{v_p(x_i)}}+\vastl(\sum_{i=1}^{m-1}p^{V_p-v_p(x_i)}\prod_{\substack{j=1\\j\neq i}}^{m-1}\frac{x_j}{p^{v_p(x_j)}}\vastr)\frac{x_m}{p^{V_p}}
\]
is divisible by $p^{V_p}$. This can happen if and only if the equation
\[
\vastl(\sum_{i=1}^{m-1}p^{V_p-v_p(x_i)}\prod_{\substack{j=1\\j\neq i}}^{m-1}\frac{x_j}{p^{v_p(x_j)}}\vastr)x_{m}'\equiv -\prod_{i=1}^{m-1}\frac{x_i}{p^{v_p(x_i)}}\bmod{p^{V_p}}
\]
has a solution for~$x_m'$. Since $-\prod_{i=1}^{m-1}(x_i/v_p(x_i))$ is not divisible by~$p$, this equation has a solution if and only if $p$ does not divide~\eqref{eq:extremalvaluationcondition}. Moreover, by adding a multiple of $p^{V_p}$ if necessary, this solution can be taken to be a odd positive integer, in which case $x_m=p^{V_p}x_m'$ is also a odd positive integer.
\end{proof}

We give an example where this proposition can be used to show that the bounds of Theorem~\ref{thm:commonfactorbounds} cannot be achieved.

\begin{example}
Suppose $m=3$, $x_1=5$, and $x_2=45$. For $p=5$, we have 
\[
\sum_{i=1}^{2}5^{V_5-v_5(x_i)}\prod_{\substack{j=1\\j\neq i}}^{2}\frac{x_j}{5^{v_5(x_j)}}=(5^{1-1})(9)+(5^{1-1})(1)=10,
\]
which is divisible by~$5$. Therefore Proposition~\ref{prop:extremalvaluationlocal} says there is no $x_3$ for which $v_5(y)=1$ and no $x_3$ for which $v_5(y)=3$.
\end{example}

The following lemma is useful to convert the statement of Proposition~\ref{prop:extremalvaluationlocal} for one prime number into a single statement for all relevant prime numbers. 

\begin{lemma}\label{lem:modularsolutions}
If $p_1,\dotsc,p_n$ are distinct prime numbers, $e_1,\dotsc,e_n$ are positive integers, and $a_1,\dotsc,a_n$ are integers, then there exists a positive integer~$x$ such that $x\equiv a_i\bmod{p_i^{e_i}}$ for all $i\in[n]$.
\end{lemma}

\begin{proof}
We inductively construct integers~$b_k$ such that $b_k\equiv a_i\bmod{p_i^{e_i}}$ for all $i\in[k]$. Let $b_1=a_1$. Given $b_k$, the equation $b_k+p_1^{e_1}\dotsm p_k^{e_k} z\equiv a_{k+1}\bmod{p_{k+1}^{e_{k+1}}}$ can be solved for~$z$ because $\gcd(p_1^{e_1}\dotsm p_k^{e_k}, p_{k+1}^{e_{k+1}})=1$. Letting $b_{k+1}=b_k+p_1^{e_1}\dotsm p_k^{e_k} z$, we have $b_{k+1}\equiv b_k\equiv a_i\bmod{p_i^{e_i}}$ for all $i\in[k]$ and $b_{k+1}\equiv a_{k+1}\bmod{p_{k+1}^{e_{k+1}}}$. Then $b_n$ has the desired property, so we can let $x=b_n$.
\end{proof}

We can now determine when the bounds of Theorem~\ref{thm:commonfactorbounds} are achieved for all relevant prime numbers. 

\begin{theorem}\label{thm:extremalvaluationglobal}
Let $x_1,\dotsc,x_{m-1}$ be odd positive integers. Given a prime number~$p$, let $V_p=\max_{i\in[m-1]}\{v_p(x_i)\}$. The following statements are equivalent:
\begin{enumerate}[label=\textup{(\alph*)},ref=\textup{\alph*}]
\item There exists an odd positive integer~$x_m$ such that $v_p(y)=v_p(x_1\dotsm x_{m-1})-V_p$ for every prime number~$p$ that divides $x_i$ for some $i\in[m-1]$. \label{item:lowervaluationglobal}
\item There exists an odd positive integer~$x_m$ such that $v_p(y)=v_p(x_1\dotsm x_{m-1})+V_p$ for every prime number~$p$ that divides $x_i$ for some $i\in[m-1]$. \label{item:uppervaluationglobal}
\item For every prime number~$p$ that divides $x_i$ for some $i\in[m-1]$, we have that $p$ does not divide \label{item:valuationcriterionglobal}
\begin{equation}\label{eq:valuationcriterionglobal}
\sum_{i=1}^{m-1}p^{V_p-v_p(x_i)}\prod_{\substack{j=1\\j\neq i}}^{m-1}\frac{x_j}{p^{v_p(x_j)}}.
\end{equation}
\end{enumerate}
\end{theorem}

\begin{proof} 
To see that~(\ref{item:lowervaluationglobal}) implies~(\ref{item:valuationcriterionglobal}), observe that if $v_p(y)=v_p(x_1\dotsm x_{m-1})-V_p$ then Proposition~\ref{prop:extremalvaluationlocal} gives that $p$ does not divide~\eqref{eq:valuationcriterionglobal}. To see that~(\ref{item:uppervaluationglobal}) implies~(\ref{item:valuationcriterionglobal}), observe that if $v_p(y)=v_p(x_1\dotsm x_{m-1})+V_p$ then Proposition~\ref{prop:extremalvaluationlocal} gives that $p$ does not divide~\eqref{eq:valuationcriterionglobal}. 

To see that~(\ref{item:valuationcriterionglobal}) implies~(\ref{item:lowervaluationglobal}), observe that if $p$ does not divide~\eqref{eq:valuationcriterionglobal} then Proposition~\ref{prop:extremalvaluationlocal} gives an odd positive integer $x_{m,p}$, determined modulo $p^{2V_p}$, for which $v_p(y)=v_p(x_1\dotsm x_{m-1})-V_p$. We can then use Lemma~\ref{lem:modularsolutions} to find an odd positive integer~$x_m$ such that $x_m\equiv x_{m,p}\bmod{p^{2V_p}}$ for all prime numbers~$p$ that divide $x_i$ for some $i\in[m-1]$. This value of~$x_m$ then has the property that $v_p(y)=v_p(x_1\dotsm x_{m-1})-V_p$ for every prime number~$p$ that divides $x_i$ for some $i\in[m-1]$.

To see that~(\ref{item:valuationcriterionglobal}) implies~(\ref{item:uppervaluationglobal}), observe that if $p$ does not divide~\eqref{eq:valuationcriterionglobal} then Proposition~\ref{prop:extremalvaluationlocal} gives an odd positive integer $x_{m,p}$, determined modulo $p^{2V_p}$, for which $v_p(y)=v_p(x_1\dotsm x_{m-1})+V_p$. We can then use Lemma~\ref{lem:modularsolutions} to find an odd positive integer~$x_m$ such that $x_m\equiv x_{m,p}\bmod{p^{V_p}}$ for all prime numbers~$p$ that divide $x_i$ for some $i\in[m-1]$. This value of~$x_m$ then has the property that $v_p(y)=v_p(x_1\dotsm x_{m-1})+V_p$ for every prime number $p$ that divides $x_i$ for some $i\in[m-1]$.
\end{proof}

\section{Length \texorpdfstring{$m$}{m} with fixed \texorpdfstring{$x_1,\dotsc,x_{m-2}$}{x\_1,...,x\_{m-2}}}\label{sec:fixedm-2}

In this section, we find infinite families of rational numbers whose odd greedy expansion has length~$m$ and begins with $m-2$ fixed compatible positive odd denominators. We begin with a lemma that shows certain bounds from~\eqref{eq:x_kbound} decrease as $x_{k-1}$ increases from one odd positive integer to another.

\begin{lemma}\label{lem:decreasing}
Let $i$, $k$, $m$, and~$t$ be positive integers with $i\leq k-2$ and $k\leq m$. If $x_1,\dotsc,x_m$ are the denominators of an odd greedy expansion, then
\[
\frac{(x_i-2)x_{i}\dotsm x_{k-1}}{2\sigma_{k-i-1}(x_i,\dotsc, x_{k-1})-x_i^2\sigma_{k-i-2}(x_{i+1},\dotsc,x_{k-1})}
\]
is greater than
\[
\frac{(x_i-2)x_{i}\dotsm x_{k-2}(x_{k-1}+2t)}{2\sigma_{k-i-1}(x_i,\dotsc, x_{k-2},x_{k-1}+2t)-x_i^2\sigma_{k-i-2}(x_{i+1},\dotsc, x_{k-2},x_{k-1}+2t)}.
\]
\end{lemma}

\begin{proof}
We first show that the denominators of the above fractions are both positive. Since $x_1,\dotsc,x_m$ are the denominators of an odd greedy expansion, Proposition~\ref{prop:expansionequivalence} gives that~\eqref{eq:x_kbound} always holds when $i<j\leq m$, where $j$ plays the role of~$k$ in Proposition~\ref{prop:expansionequivalence}. In particular, \eqref{eq:x_kbound} holds with $m$ replaced by $k-1$, so $x_1,\dotsc,x_{k-1}$ are also the denominators of an odd greedy expansion. Proposition \ref{prop:expansionequivalence}(\ref{item:inequality}) with $j=k-1$ then gives that 
\[
2\sigma_{k-i-1}(x_i,\dotsc, x_{k-1})-x_i^2\sigma_{k-i-2}(x_{i+1},\dotsc, x_{k-1})
\]
is positive. Moreover, since when $m=k-1$ we have that $x_{k-1}$ only appears on the left of~\eqref{eq:x_kbound}, this inequality still holds when $x_{k-1}$ is replaced by $x_{k-1}+2t$. Thus $x_1,\dotsc,x_{k-2},x_{k-1}+2t$ are the denominators of an odd greedy expansion as well. Proposition \ref{prop:expansionequivalence}(\ref{item:inequality}) with $j=k-1$ then gives that 
\[
2\sigma_{k-i-1}(x_i,\dotsc, x_{k-2},x_{k-1}+2t)-x_i^2\sigma_{k-i-2}(x_{i+1},\dotsc, x_{k-2},x_{k-1}+2t)
\]
is positive. 

Multiplying through by the denominators of the fractions in the statement, we see that it is sufficient to show
\begin{equation}\label{eq:crossproduct1}
\begin{split}
&(x_i-2)x_{i}\dotsm x_{k-1}\\
&\quad\cdot\bigl(2\sigma_{k-i-1}(x_i,\dotsc, x_{k-2},x_{k-1}+2t)-x_i^2\sigma_{k-i-2}(x_{i+1},\dotsc, x_{k-2},x_{k-1}+2t)\bigr)
\end{split}
\end{equation}
is greater than
\begin{equation}\label{eq:crossproduct2}
(x_i-2)x_{i}\dotsm x_{k-2}(x_{k-1}+2t)\bigl(2\sigma_{k-i-1}(x_i,\dotsc,x_{k-1})-x_i^2\sigma_{k-i-2}(x_{i+1},\dotsc, x_{k-1})\bigr).
\end{equation}
Expanding and distributing, we have that~\eqref{eq:crossproduct1} is equal to 
\[
\begin{split}
&2(x_i-2)(x_{i}\dotsm x_{k-2})^2x_{k-1}+2(x_i-2)x_{i}\dotsm x_{k-2}x_{k-1}^2\sigma_{k-i-2}(x_i,\dotsc, x_{k-2})\\
&\quad+4t(x_i-2)x_{i}\dotsm x_{k-1}\sigma_{k-i-2}(x_i,\dotsc, x_{k-2})-(x_i-2)x_i^3(x_{i+1}\dotsm x_{k-2})^2x_{k-1}\\
&\quad-(x_i-2)x_{i}^3x_{i+1}\dotsm x_{k-2}x_{k-1}^2\sigma_{k-i-3}(x_i,\dotsc, x_{k-2})\\
&\quad-2t(x_i-2)x_{i}^3x_{i+1}\dotsm x_{k-1}\sigma_{k-i-3}(x_i,\dotsc, x_{k-2}).
\end{split}
\]
Similarly, \eqref{eq:crossproduct2} is equal to 
\[
\begin{split}
&2(x_i-2)(x_{i}\dotsm x_{k-2})^2x_{k-1}+2(x_i-2)x_{i}\dotsm x_{k-2}x_{k-1}^2\sigma_{k-i-2}(x_i,\dotsc, x_{k-2})\\
&\quad+4t(x_i-2)(x_{i}\dotsm x_{k-2})^2+4t(x_i-2)x_{i}\dotsm x_{k-1}\sigma_{k-i-2}(x_i,\dotsc, x_{k-2})\\
&\quad-(x_i-2)x_i^3(x_{i+1}\dotsm x_{k-2})^2x_{k-1}\\
&\quad-(x_i-2)x_{i}^3x_{i+1}\dotsm x_{k-2}x_{k-1}^2\sigma_{k-i-3}(x_i,\dotsc, x_{k-2})-2t(x_i-2)x_i^3(x_{i+1}\dotsm x_{k-2})^2\\
&\quad-2t(x_i-2)x_{i}^3x_{i+1}\dotsm x_{k-1}\sigma_{k-i-3}(x_i,\dotsc, x_{k-2}).
\end{split}
\]
Comparing terms, we see that~\eqref{eq:crossproduct2} is~\eqref{eq:crossproduct1} plus
\[
4t(x_i-2)(x_{i}\dotsm x_{k-2})^2-2t(x_i-2)x_i^3(x_{i+1}\dotsm x_{k-2})^2
=-2t(x_i-2)^2(x_{i}\dotsm x_{k-2})^2.
\]
Since $t$ is positive and each $x_j$ is odd, this is always negative, so the result follows.
\end{proof}

Given compatible denominators $x_1,\dotsc,x_{m-2}$, we can now find an infinite family of rational numbers whose odd greedy expansion has length~$m$ and begins with denominators $x_1,\dotsc,x_{m-2}$.

\begin{theorem}\label{thm:fixedm-2}
Let $m\geq3$, and suppose $x_1,\dotsc,x_{m-2}$ are odd positive integers that satisfy~\eqref{eq:x_kbound} whenever $i<k\leq m-2$. Let $c_1$ be a nonnegative integer such that
\[
\frac{x_{m-2}^2+3}{2}-x_{m-2}+2c_1>\frac{(x_i-2)x_i\dotsm x_{m-2}}{2\sigma_{m-i-2}(x_i,\dotsc, x_{m-2})-x_i^2\sigma_{m-i-3}(x_{i+1},\dotsc, x_{m-2})}
\]
for all $i\in[m-3]$, let $b=(x_{m-2}^2+3)/2-x_{m-2}+2c_1$, and let $c_2$ be a nonnegative integer such that
\[
\frac{b^2+3}{2}-b+2c_2>\frac{(x_i-2)x_i\dotsm x_{m-2}b}{2\sigma_{m-i-1}(x_i,\dotsc, x_{m-2},b)-x_i^2\sigma_{m-i-2}(x_{i+1},\dotsc, x_{m-2},b)}
\]
for all $i\in[m-2]$. For all nonnegative integers $t_1$ and~$t_2$, the rational number 
\[
\frac{\sigma_{m-1}(x_1,\dotsc,x_m)}{x_1\dotsm x_m}
\]
with $x_{m-1}=b+2t_1$ and
\[
x_m=\frac{x_{m-1}^2+3}{2}-x_{m-1}+2c_2+2t_2,
\]
has odd greedy expansion of length~$m$ with denominators $x_1,\dotsc,x_m$.
\end{theorem}

\begin{proof}
According to Proposition~\ref{prop:expansionequivalence}, we need to show that~\eqref{eq:x_kbound} holds for all positive integers $i$ and~$k$ with $i<k\leq m$. Since we have assumed~\eqref{eq:x_kbound} holds whenever $k\leq m-2$, it only remains to show that it also holds for $k=m-1$ and $k=m$.

We have 
\[
\begin{split}
x_{m-1}&=b+2t_1\\
&\geq\frac{x_{m-2}^2+3}{2}-x_{m-2}+2c_1\\
&>\frac{(x_i-2)x_i\dotsm x_{m-2}}{2\sigma_{m-i-2}(x_i,\dotsc, x_{m-2})-x_i^2\sigma_{m-i-3}(x_{i+1},\dotsc, x_{m-2})}
\end{split}
\]
for all $i\in[m-2]$, so~\eqref{eq:x_kbound} holds when $k=m-1$.

Since $(x_{m-1}^2+3)/2-x_{m-1}+2c_2$ is increasing in $x_{m-1}$ for $x_{m-1}\geq1$, we have 
\[
\begin{split}
x_m&=\frac{x_{m-1}^2+3}{2}-x_{m-1}+2c_2+2t_2\\
&\geq\frac{(b+2t_1)^2+3}{2}-(b+2t_1)+2c_2\\
&\geq\frac{b^2+3}{2}-b+2c_2\\
&>\frac{(x_i-2)x_i\dotsm x_{m-2}b}{2\sigma_{m-i-1}(x_i,\dotsc, x_{m-2},b)-x_i^2\sigma_{m-i-2}(x_{i+1},\dotsc, x_{m-2},b)}\\
&\geq\frac{(x_i-2)x_i\dotsm x_{m-1}}{2\sigma_{m-i-1}(x_i,\dotsc, x_{m-1})-x_i^2\sigma_{m-i-2}(x_{i+1},\dotsc, x_{m-1})},
\end{split}
\]
where the last inequality follows from Lemma~\ref{lem:decreasing}, Therefore~\eqref{eq:x_kbound} holds when $k=m$, completing the proof.
\end{proof}

We now record the special case of Theorem~\ref{thm:fixedm-2} when $m=3$.

\begin{corollary}\label{cor:fixed1length3}
Let $x_1$ be an odd positive integer, let $b=(x_{1}^2+3)/2-x_{1}$, and let $c_2$ be a nonnegative integer such that
\[
\frac{b^2+3}{2}-b+2c_2>\frac{(x_1-2)x_1b}{2(x_1+b)-x_1^2}.
\]
For all nonnegative integers $t_1$ and~$t_2$, the rational number 
\[
\frac{\sigma_{2}(x_1,x_2,x_3)}{x_1x_2x_3}
\]
with $x_{2}=b+2t_1$ and
\[
x_3=\frac{x_{2}^2+3}{2}-x_{2}+2c_2+2t_2,
\]
has odd greedy expansion of length~$3$ with denominators $x_1,x_2,x_3$.
\end{corollary}

\begin{proof}
When $m=3$, we can take $c_1=0$ in Theorem~\ref{thm:fixedm-2} so that $b=(x_{1}^2+3)/2-x_{1}$ and the only condition on~$c_2$ is when $i=1$. 
\end{proof}

We give two examples.

\begin{example}\label{ex:fixed1length3a}
Suppose $m=3$ and $x_1=5$. In Corollary~\ref{cor:fixed1length3}, we have
\[
b=\frac{x_1^2+3}{2}-x_1=\frac{5^2+3}{2}-5=9.
\]
We also have
\[
\frac{b^2+3}{2}-b=\frac{9^2+3}{2}-9=33
\]
and
\[
\frac{(x_1-2)x_1b}{2(x_1+b)-x_1^2}=\frac{(5-2)5\cdot9}{2(5+9)-5^2}=45.
\]
Therefore we can take $c_2=7$. We then have $x_{2}=9+2t_1$ and
\[
\begin{split}
x_3&=\frac{x_{2}^2+3}{2}-x_{2}+2c_2+2t_2\\
&=\frac{(9+2t_1)^2+3}{2}-(9+2t_1)+14+2t_2\\
&=47+16t_1+2t_1^2+2t_2.
\end{split}
\]
Corollary~\ref{cor:fixed1length3} then gives that, for all nonnegative integers $t_1$ and~$t_2$, 
\[
\begin{split}
\frac{\sigma_2(x_1,x_2,x_3)}{x_1x_2x_3}&=\frac{\sigma_2(5,9+2t_1,47+16t_1+2t_1^2+2t_2)}{5(9+2t_1)(47+16t_1+2t_1^2+2t_2)}\\
&=\frac{703+328t_1+60t_1^2+4t_1^3+4t_1t_2+28t_2}{2115+1190t_1+250t_1^2+20t_1^3+20t_1t_2+90t_2}
\end{split}
\]
has odd greedy expansion of length~$3$ with denominators $x_1,x_2,x_3$. In the special case when $t_1=0$, we recover the family of fractions from Example~\ref{ex:fixed2length3}.
\end{example}

\begin{example}\label{ex:fixed1length3b}
Suppose $m=3$ and $x_1=3$. In Corollary~\ref{cor:fixed1length3}, we have
\[
b=\frac{x_1^2+3}{2}-x_1=\frac{3^2+3}{2}-3=3.
\]
We also have
\[
\frac{b^2+3}{2}-b=\frac{3^2+3}{2}-3=3
\]
and
\[
\frac{(x_1-2)x_1b}{2(x_1+b)-x_1^2}=\frac{(3-2)3\cdot3}{2(3+3)-3^2}=3.
\]
Therefore we can take $c_2=1$. We then have $x_{2}=3+2t_1$ and
\[
\begin{split}
x_3&=\frac{x_{2}^2+3}{2}-x_{2}+2c_2+2t_2\\
&=\frac{(3+2t_1)^2+3}{2}-(3+2t_1)+2+2t_2\\
&=5+4t_1+2t_1^2+2t_2.
\end{split}
\]
Corollary~\ref{cor:fixed1length3} then gives that, for all nonnegative integers $t_1$ and~$t_2$, 
\[
\begin{split}
\frac{\sigma_2(x_1,x_2,x_3)}{x_1x_2x_3}&=\frac{\sigma_2(3,3+2t_1,5+4t_1+2t_1^2+2t_2)}{3(3+2t_1)(5+4t_1+2t_1^2+2t_2)}\\
&=\frac{39+40t_1+20t_1^2+4t_1^3+4t_1t_2+12t_2}{45+66t_1+42t_1^2+12t_1^3+12t_1t_2+18t_2}
\end{split}
\]
has odd greedy expansion of length~$3$ with denominators $x_1,x_2,x_3$. In the special case when $t_1=1$, this becomes
\[
\frac{103+16t_2}{165+30t_2}.
\]

This agrees with Example~\ref{ex:fixed2length3a} by letting $t=t_2+1$, except that the fraction $87/135$ arises in Example~\ref{ex:fixed2length3a} but not here. Unlike in Section~\ref{sec:fixedm-1}, where we found \emph{all} rational numbers whose odd greedy expansion has length~$m$ and begins with denominators $x_1,\dotsc,x_{m-1}$, in this section we may not find all rational numbers whose odd greedy expansion has length~$m$ and begins with denominators $x_1,\dotsc,x_{m-2}$. 
\end{example}

\section*{Acknowledgments}

We would like to thank Nathan Kaplan for suggesting we study the odd greedy algorithm and Claire Levaillant for helpful conversations. Joel Louwsma was partially supported by a Niagara University Summer Research Award. Joseph Martino was partially supported by a Niagara University Undergraduate Student--Faculty Research Collaboration Award.

\bibliographystyle{amsplain}
\bibliography{OddGreedyLength.bib}

\providecommand{\bysame}{\leavevmode\hbox to3em{\hrulefill}\thinspace}
\providecommand{\MR}{\relax\ifhmode\unskip\space\fi MR }
\providecommand{\MRhref}[2]{%
  \href{http://www.ams.org/mathscinet-getitem?mr=#1}{#2}
}
\providecommand{\href}[2]{#2}
\begin{thebibliography}{10}

\bibitem{Breusch}
R.~Breusch, \emph{A special case of {E}gyptian fractions, solution to advanced problem 4512}, Amer. Math. Monthly \textbf{61} (1954), no.~3, 200–201.

\bibitem{Brown}
K.~Brown, \emph{Odd-greedy unit fraction expansions}, \\\url{https://www.mathpages.com/home/kmath478.htm} (accessed Sept.\ 13, 2023).

\bibitem{Eppstein}
D.~Eppstein, \emph{Algorithms for {E}gyptian fractions}, \\\url{https://www.ics.uci.edu/~eppstein/numth/egypt/approx.html} (accessed Sept.\ 13, 2023).

\bibitem{Fibonacci}
L.~Fibonacci, \emph{Fibonacci's {L}iber abaci: a translation into modern {E}nglish of {L}eonardo {P}isano's {B}ook of calculation}, Springer-Verlag, New York, 2002, Translated from the Latin and with an introduction, notes, and bibliography by L.~Sigler. \MR{1923794}

\bibitem{Guy2004}
R.~K. Guy, \emph{Unsolved problems in number theory}, third ed., Problem Books in Mathematics, Springer-Verlag, New York, 2004. \MR{2076335}

\bibitem{KW}
V.~Klee and S.~Wagon, \emph{Old and new unsolved problems in plane geometry and number theory}, The Dolciani Mathematical Expositions, vol.~11, Mathematical Association of America, Washington, DC, 1991. \MR{1133201}

\bibitem{Pihko2010}
J.~Pihko, \emph{Remarks on the ``greedy odd'' {E}gyptian fraction algorithm {II}}, Fibonacci Quart. \textbf{48} (2010), no.~3, 202--208. \MR{2722216}

\bibitem{RS}
G.~Robins and C.~Shute, \emph{The {R}hind mathematical papyrus: an ancient {E}gyptian text}, British Museum Publications, Ltd., London, 1987. \MR{910500}

\bibitem{Starke}
E.~P. Starke, \emph{Advanced problem 4512}, Amer. Math. Monthly \textbf{59} (1952), no.~9, 640.

\bibitem{Stewart}
B.~M. Stewart, \emph{Sums of distinct divisors}, Amer. J. Math. \textbf{76} (1954), no.~4, 779--785. \MR{64800}

\bibitem{Sylvester}
J.~J. Sylvester, \emph{On a point in the theory of vulgar fractions}, Amer. J. Math. \textbf{3} (1880), no.~4, 332--335. \MR{1505274}

\bibitem{Wagon}
S.~Wagon, \emph{Mathematica\textsuperscript{®} in action}, second ed., Springer-Verlag, New York, 1999.

\end{thebibliography}

\end{document}